%% file: main.tex
\documentclass[11 pt,colorinlistoftodos]{article}

\input{preambles/preamble}
\input{preambles/mathpreamble}

\usetikzlibrary{decorations.markings}
\usepackage{amsmath}
\usepackage{mathtools}
\usepackage{gensymb}
\usepackage{tikz-3dplot}

\title{\normalsize \textbf{Strata of toric hyperplane arrangements, zonotope lattice points, and the Bondal-Thomsen collection}}
\author{\normalsize Friedrich Bauermeister, Andrew Hanlon, Davis Painter, Sair Shaikh, Benjamin Singer}
\date{}

\newcommand{\Addresses}{{
  \bigskip
  \footnotesize

  \noindent F.~Bauermeister, \textsc{Department of Mathematics, Dartmouth College}\par\nopagebreak
  \noindent \textit{E-mail address}: \hyperlink{mailto:friedrich.bauermeister.gr@dartmouth.edu}{\texttt{friedrich.bauermeister.gr@dartmouth.edu}}

  \medskip

  \noindent A.~Hanlon, \textsc{Department of Mathematics, University of Oregon}\par\nopagebreak
  \noindent \textit{E-mail address}: \hyperlink{mailto:ahanlon@uoregon.edu}{\texttt{ahanlon@uoregon.edu}}

  \medskip

  \noindent D.~Painter, \textsc{Dartmouth College}\par\nopagebreak
  \noindent \textit{E-mail address}: \hyperlink{mailto:davis.g.painter.26@dartmouth.edu}{\texttt{davis.g.painter.26@dartmouth.edu}}

  \medskip

  \noindent S.~Saikh, \textsc{Dartmouth College}\par\nopagebreak
  \noindent \textit{E-mail address}: \hyperlink{mailto:sair.shaikh.26@dartmouth.edu}{\texttt{sair.shaikh.26@dartmouth.edu}}

  \medskip

  \noindent B.~Singer, \textsc{Dartmouth College}\par\nopagebreak
  \noindent \textit{E-mail address}: \hyperlink{mailto:benjamin.d.singer.27@dartmouth.edu}{\texttt{benjamin.d.singer.27@dartmouth.edu}}

  \medskip

}}

\begin{document}

\maketitle

\begin{abstract}
    We show that strata of oriented toric hyperplane arrangements are in bijection with a collection of lattice points in a zonotope. 
    Moreover, we relate the dimension of the stratum and the dimension of the minimal face of the zonotope containing the corresponding lattice point.
    We discuss how this correspondence is related to toric varieties and the Bondal-Thomsen generators of their derived categories.
\end{abstract}

\section{Introduction}

This paper concerns two objects of discrete geometry: a zonotope and a toric hyperplane arrangement.
Zonotopes are a special class of convex polytopes obtained as Minkowski sums of line segments or, equivalently, by projecting a cube under a linear map.
Zonotopes are a classically studied object of convex geometry (see, e.g., \cite{grunbaum1967convex,ziegler2012lectures}) with foundational work on their structure found in \cite{mcmullen1971zonotopes,shephard1974combinatorial,billera1992fiber} among many other places.
On the other hand, toric hyperplane arrangements have received considerably less attention. 
In this paper, a toric hyperplane arrangement consists of a collection of codimension one linear subtori in a real torus. 
Some combinatorial properties of toric hyperplane arrangements have been studied in \cite{ehrenborg2009affine, lawrence2011enumeration, chandrasekhar2017face, bergerova2023symmetry, hanlon2024rational}.

Zonotopes are closely related to central hyperplane arrangements, but we are interested in a more recent interaction with toric hyperplane arrangements that has appeared in the study of dervied categories of toric varieties and homological mirror symmetry. 
In an influential note \cite{bondal2006derived}, Bondal proposed that there are natural line bundle generators of the derived category of a smooth toric variety corresponding to strata of a toric hyperplane arrangement. 
Earlier, these same line bundles were realized as the summands of the pushforward of the structure sheaf under toric Frobenius and indexed by certain lattice points of a zonotope \cite{thomsen2000frobenius} (see also \cite{bogvad1998splitting,achinger2015characterization}).  
Bondal's proposal was recently realized in \cite{favero2023rouquier,hanlon2024resolutions,ballard2024king} where the generating set of line bundles is called the Thomsen or Bondal-Thomsen collection.
Homological mirror symmetry for toric varieties has been heavily influenced by and has built upon Bondal's proposal, see, e.g., \cite{fang2011categorification,fang2012t,hanlon2022aspects}.
Our main result explores this correspondence between lattice points of a zonotope and the strata of a toric hyperplane arrangement from a purely combinatorial viewpoint and, as a consequence, provides new insight into these generators. 

We will set up and state our main combinatorial result in \cref{subsec:combintro} and discuss how it is related to the Bondal-Thomsen collection in \cref{subsec:toricintro}.

\subsection{Combinatorial setup and result} \label{subsec:combintro}
Before stating our main result, we make the set up precise.
Let $A = \{v_1, \hdots v_k \}$ be a finite ordered subset of $\ZZ^n$. 
We first consider the map $
\varphi \colon \RR^n \to \RR^k $
given by setting the $j$th component of $\varphi(u)$ to be $\langle u, v_j \rangle$. 
Here, $\langle \cdot, \cdot \rangle$ is the standard inner product or dot product on $\RR^n$.
More abstractly, we could consider $v_1, \hdots, v_k$ to lie in a lattice $N \simeq \ZZ^n$ and the domain of $\varphi$ to be $M \otimes_{\ZZ} \RR$ where $M$ is the dual lattice. 
We will largely favor the more concrete approach outside of later discussions around toric geometry.

We set the cokernel of $\varphi$ to be $\pi \colon \RR^k \to \RR^k/\mathrm{im}(\varphi)$ from which we can define the \emph{half-open zonotope}
\begin{equation} \label{eq:zono}
    Z = \pi([0,1)^k)
\end{equation}
whose closure is the zonotope $\bar{Z} = \pi([0,1]^k)$. 
Note that $\bar{Z}$ is a lattice polytope with respect to the lattice $\pi(\ZZ^k)$. 
In addition, observe that $Z$ is convex so the intersection of any two faces of $Z$ is a face of $Z$. 
This implies that for every point $p \in Z$ there is a face of smallest dimension containing $p$ which we will call the minimal face of $Z$ containing $p$.

For our other construction, we consider the quotient $\RR^n/\varphi^{-1}(\ZZ^k)$ and the maps
\[ h_j \colon \RR^n/\varphi^{-1}(\ZZ^k) \to \RR/ \ZZ \]
induced by $\langle \cdot , v_j \rangle $. We then obtain an oriented hyperplane arrangement given by the $h_j^{-1}(0)$.
We consider the following induced stratification of $\RR^n/\varphi^{-1}(\ZZ^k)$. Let $\Phi \colon \RR^n \to \ZZ^k$ be the map such that the $j$th coordinate of $\Phi(u)$ is given by $\lceil \langle u, v_j \rangle \rceil$ where $\lceil \cdot \rceil$ 
is the ceiling function. 
For every $c \in \ZZ^k$, we obtain a region $\wt{S}_c \coloneq \Phi^{-1}(c)$ of $\RR^n$, and we obtain a stratification of $\RR^n/\varphi^{-1}(\ZZ^k)$ by the images $S_c$ of these sets under the quotient map.
We call these the $\Phi$-strata of $\RR^n/\varphi^{-1}(\ZZ^k)$.

 Note that the boundaries of the $S_c$ lie on the hyperplanes $h_j^{-1}(0)$ and these boundary pieces may or may not be part of a $\Phi$-stratum depending on the orientation.\footnote{In particular, the $\Phi$-strata may be larger than the typical equidimensional strata one associates to an unoriented hyperplane arrangement such as in \cite{ehrenborg2009affine,hanlon2024rational}.} 
 In fact, it will be important to record the set of hyperplanes entirely containing a stratum.
Thus, for a $\Phi$-stratum $S$, we define 
\begin{equation} \label{eq:js}
    J_S \coloneq \left\{ j \in \{1, \cdots, k\} : S \subseteq h_j^{-1}(0) \right\}.
\end{equation} 
The main result of this paper relates these objects as follows.

\begin{mainthm} \label{thm:main}
    There is a bijection between the $\Phi$-strata of $\RR^n/\varphi^{-1}(\ZZ^k)$ and the lattice points of $Z$. 
    Moreover, the assignment of a minimal face $F_p$ of $Z$ containing the point $p \in Z$ corresponding to a stratum $S$ is entirely determined by $J_S$, is inclusion-reversing with respect to $J_S$, and 
    \begin{equation} \label{eq:dimension} \dim F_p + \dim S = k - | J_S | + \dim\left( \ker \varphi \cap \mathrm{span}(\wt{S}-u)\right) 
    \end{equation}
    for any lift $\wt{S}$ of $S$ to $\RR^n$ and any $u \in \wt{S}$.
\end{mainthm}

\begin{rem} \cref{thm:main} simplifies in most cases of interest to us. 
Namely, if $A$ contains a basis of $\RR^n$, then $\ker \varphi = 0$ eliminating a term from \cref{eq:dimension}. 
Further, if $A$ contains a basis of $\ZZ^n$, i.e., a list of $n$ vectors whose determinant is $\pm 1$, then $\varphi^{-1}(\ZZ^k) = \ZZ^n$ so our $\Phi$-strata are subsets of the torus $\RR^n/\ZZ^n$.  
\end{rem}

\begin{rem}
    We note that every lattice zonotope can be obtained from the construction above in the following sense.
    Suppose that  $\bar{Z} \subseteq \RR^m$ is the Minkowski sum of intervals $[0, w_j]$ with $w_1, \hdots, w_k \in \ZZ^m$.
    Then, $\bar{Z} = \pi([0,1]^k)$ where $\pi(e_j) = w_j$. 
    Since $\pi$ is represented by a matrix with integral entries, we can choose a basis $u_1, \hdots, u_n \in \ZZ^k$ of $\ker(\pi)$. 
    Then, we can define $\varphi$ by setting the $v_j$ to be the row vectors of the matrix whose column vectors are $u_1, \hdots, u_n$.
\end{rem}

\begin{rem}
    There is a simple formula due to Stanley \cite{stanley1990zonotope} to calculate the lattice points in the zonotope $\bar{Z}$. 
    There are also explicit formulae for the number of strata in a toric hyperplane stratification in \cite{ehrenborg2009affine,lawrence2011enumeration}.
    It would be interesting to relate these results to \cref{thm:main} and thus possibly to each other.
    In general, some modification is needed as \cref{thm:main} concerns only lattice points in $Z$ and $\Phi$-strata can contain multiple regions in the sense of \cite{ehrenborg2009affine,lawrence2011enumeration}.
    However, such a relation should be more straightforward between interior lattice points of $\bar{Z}$ and top-dimensional strata/regions.
\end{rem}

The proof of \cref{thm:main} is contained in \cref{sec:proof}. 
The main technical tool is \cref{lem:relint}, which is a classical result in convex analysis. 
In \cref{sec:faces}, we refine \cref{thm:main} by showing that faces of $Z$ are half-open zonotopes that come from restricting our construction to the linear space spanned by vectors in a lift of $S$.
Before briefly discussing how \cref{thm:main} relates to the Bondal-Thomsen collection in \cref{subsec:toricintro}, we give two examples of \cref{thm:main}.  

\begin{example} \label{ex:hirz2}
Set $A = \{e_1, e_2, -e_1 + 2e_2, -e_2 \}$ where $e_1, e_2$ are the standard basis vectors of $\mathbb{R}^2$. In this example, there are three strata of dimension two, one stratum of dimension one, and one stratum of dimension zero corresponding to lattice points in the zonotope lying on faces of the same dimension. See Figure 1. In this example, $\varphi$ is given by $\varphi(x,y) = (x,y,-x+2y,-y)$ and $\pi$ is given by $\pi(w_1, w_2, w_3, w_4) = (w_3 + w_1 - 2w_2, w_4 + w_2)$. In toric geometry, $A$ is the list of primitive generators for a fan of the Hirzebruch surface of type two (and also other toric varieties whose fans are generated by the same rays with different two-dimensional cones).
\end{example}

\begin{figure}

\begin{center}
    \begin{tikzpicture}[decoration=border]

    \begin{scope}[scale = 1.25]
    
    \draw [thick, postaction={decorate}]
        (-4,1.5) -- (-1,1.5);
    \draw [thick, postaction={decorate}]
        (-2.5,0) -- (-2.5,3);
    \draw[black, ultra thick, ->](-2.5,1.5) -- (-1.8,1.5);
    \draw[black, ultra thick, ->](-2.5,1.5) -- (-2.5,2.2);
    \draw[black, ultra thick, ->](-2.5,1.5) -- (-2.5,0.8);
    \draw[black, ultra thick, ->](-2.5,1.5) -- (-3.2,2.9);
    \draw [thick, postaction={draw,decorate, black, decoration={border, amplitude=0.09cm, angle=90, segment length = 0.25cm}}]
        (0,0) -- (3,1.5);
    \draw [thick, postaction={draw,decorate, black, decoration={border, amplitude=0.09cm, angle=90, segment length = .25cm}}]
        (0,1.5) -- (3,3);
    \draw [thick, postaction={draw,decorate, black, decoration={border, amplitude=0.09cm, angle=90, segment length = .25cm}}]
        (0,3) -- (0,0);
    \draw [thick, postaction={draw,decorate, black, decoration={border, amplitude=0.09cm, angle=90, segment length = .25cm}}]
        (3,3) -- (3,0);
    \draw [thick, postaction={draw,decorate, black, decoration={border, amplitude=0.09cm, angle=90, segment length = .25cm}}]
        (0,0) -- (3,0) -- (0,0);
    \draw [thick, postaction={draw,decorate, black, decoration={border, amplitude=0.09cm, angle=90, segment length = .25cm}}]
        (0,3) -- (3,3) -- (0,3);

    \draw [thick, postaction={decorate}]
        (7,0) -- (7,3);
    \draw [thick, postaction={decorate}]
        (4.5,0.5) -- (9.5,0.5);
    \draw [thick, gray]
        (7,0.5) -- (5,1.5) -- (5,2.5) -- (7,2.5) -- (9,1.5) -- (9,0.5) -- (7,0.5);
    \filldraw[black] (7,0.5) circle (3pt);
    \filldraw[blue!80!black] (7,1.5) circle (3pt);
    \filldraw[red] (8,0.5) circle (3pt);
    \filldraw[gray] (5,1.5) circle (2pt);
    \filldraw[blue] (6,1.5) circle (3pt);
    \filldraw[gray] (9,0.5) circle (2pt);
    \filldraw[blue!60!black] (8,1.5) circle (3pt);
    \filldraw[gray] (5,2.5) circle (2pt);
    \filldraw[gray] (6,2.5) circle (2pt);
    \filldraw[gray] (9,1.5) circle (2pt);
    \filldraw[gray] (7,2.5) circle (2pt);
    
    \filldraw[black] (0,0) circle (3pt);
    \filldraw[red] (1.5,0) circle (3pt);
    \filldraw[blue] (1.5, 0.375) circle (3pt);
    \filldraw[blue!80!black] (1.5, 1.5) circle (3pt);
    \filldraw[blue!60!black] (1.5, 2.625) circle (3pt);

    \end{scope}
    
\end{tikzpicture}
\end{center}
\caption{$A$ from \cref{ex:hirz2} depicted on the left. The corresponding toric hyperplane arrangement is depicted in the middle with ``hairs" to designate the orientations. The zonotope is depicted on the right with each of the five lattice points of $Z$ colored the same as the corresponding stratum.} 
\label{fig:p1}

\end{figure}
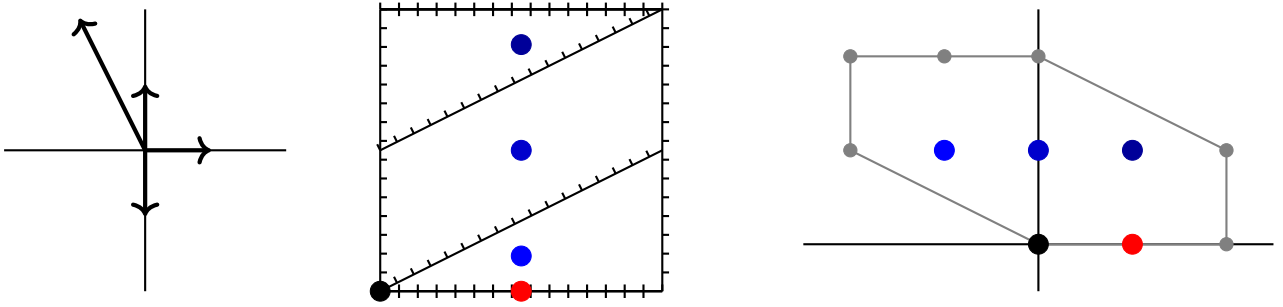

\begin{example} \label{ex:blhirz2}

Set $A = \{e_1, e_2, -e_1 + 2e_2, -e_1 - e_2,-e_2 \}$ where $e_1,e_2$ are the standard basis vectors of $\mathbb{R}^2$. In this example, there are six strata of dimension two, one stratum of dimension one, and one stratum of dimension zero. In $Z$, there are six lattice points on the dimension three face, one lattice point on a two-dimensional face, and one lattice point at at a vertex. See Figure 2. In this example, $\varphi$ is given by $\varphi(x,y) = (x, y, -x + 2y, -x - y, -y)$ and $\pi$ is given by $\pi(w_1,w_2,w_3,w_4,w_5) = (w_3 + w_1 - 2w_2, w_5+w_2, w_4 + w_1 + w_2)$. In toric geometry, $A$ is the list of primitive generators for a fan of a weighted blow-up of the Hirzebruch surface of type two at a torus fixed point. 
\end{example}

\begin{figure}

\begin{center}
    \begin{tikzpicture}[decoration=border]

    \begin{scope}[scale=1.5]
    \draw [thick, postaction={decorate}]
        (-4,1.5) -- (-1,1.5);
    \draw [thick, postaction={decorate}]
        (-2.5,0) -- (-2.5,3);
    \draw[ ultra thick, ->](-2.5,1.5) -- (-1.8,1.5);
    \draw[ ultra thick, ->](-2.5,1.5) -- (-2.5,2.2);
    \draw[ ultra thick, ->](-2.5,1.5) -- (-2.5,0.8);
    \draw[ ultra thick, ->](-2.5,1.5) -- (-3.2,2.9);
    \draw[ ultra thick, ->](-2.5,1.5) -- (-3.2,0.8);
    \draw [thick, postaction={draw,decorate, decoration={border, amplitude=0.09cm, angle=90, segment length = .25cm}}]
        (0,0) -- (3,1.5);
    \draw [thick, postaction={draw,decorate, decoration={border, amplitude=0.09cm, angle=90, segment length = .25cm}}]
        (0,1.5) -- (3,3);
    \draw [thick, postaction={draw,decorate, decoration={border, amplitude=0.09cm, angle=90, segment length = .25cm}}]
        (0,3) -- (0,0);
    \draw [thick, postaction={draw,decorate, decoration={border, amplitude=0.09cm, angle=90, segment length = .25cm}}]
        (3,3) -- (3,0);
    \draw [thick, postaction={draw,decorate, decoration={border, amplitude=0.09cm, angle=90, segment length = .25cm}}]
        (3,0) -- (0,3);
    \draw [thick, postaction={draw,decorate, decoration={border, amplitude=0.09cm, angle=90, segment length = .25cm}}]
        (0,0) -- (3,0) -- (0,0);
    \draw [thick, postaction={draw,decorate, decoration={border, amplitude=0.09cm, angle=90, segment length = .25cm}}]
        (0,3) -- (3,3) -- (0,3);
    
    \filldraw[gray] (4,0.05) -- (7,0.05) -- (7,3) -- (4,3) -- (4,0.05);
    \draw[white, ultra thick] (4,3) -- (7,0);
    \draw[white, ultra thick] (3.95,1.475) -- (7,3);
    \draw[white, ultra thick] (4,0) -- (7.05,1.525);
    \draw[white, ultra thick] (3.95,3) -- (7.05,3);
    \draw[white, ultra thick] (4.1,0.05) -- (6.95,0.05);
    \draw[red, ultra thick] (4.15, 0) -- (6.95, 0);
    \filldraw[white] (4,0) circle (2.5pt);
    \filldraw[black] (4,0) circle (1.5pt);

    \end{scope}
    
\end{tikzpicture}

\end{center}
\caption{$A$ from \cref{ex:blhirz2} is depicted on the left. The corresponding hyperplane arrangement is depicted in the middle  with ``hairs" to designate the orientations. The eight strata are depicted on the right with the two dimensional strata in gray, the one dimensional stratum in red, and the zero dimensional stratum in black.} 
\label{fig:p2}
\end{figure}
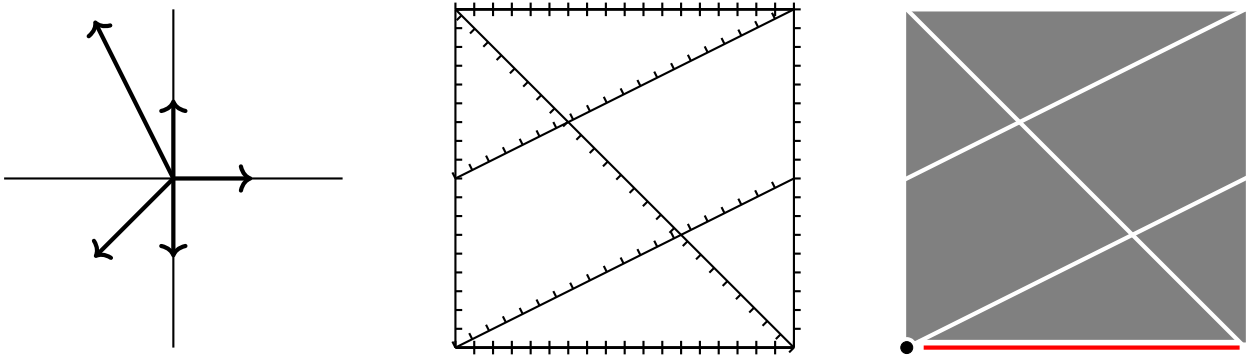

\subsection{Connection to toric geometry} \label{subsec:toricintro}

We now explain how our results apply to toric geometry. 
Suppose that $X= X_\Sigma$ is a semiprojective toric variety with simplicial fan $\Sigma$.\footnote{There is still a Bondal-Thomsen collection and an interpretation of \cref{thm:main} more generally, but we assume semiprojectivity in order to discuss the secondary fan and to simplify some statements.}
If $M$ is the character lattice of $X$, there is a map $\varphi$ from $M$ to the free abelian group on the toric divisors of $X_\Sigma$ given by pairing with the primitive generators of $\Sigma$ and whose cokernel is naturally identified with the class group of $X$. 
After tensoring with $\RR$ (and choosing a basis of $M$), we are precisely in the set-up of \cref{subsec:combintro} with $k = |\Sigma(1)|$ the number of rays of $\Sigma$. 
Then, we have a half-open zonotope $Z = Z_X$ in the real class group of $X$ defined as in \eqref{eq:zono}.
The Bondal-Thomsen collection $\Theta = \Theta_X$ consists of the elements of the class group of $X$ with image in $-Z$. 
As noted above, $\Theta$ has received considerable recent attention as the sheaves $\mathcal{O}_X(-d)$ with $-d \in \Theta$ are natural generators of the derived category and were used to solve several homological conjectures concerning toric varieties \cite{favero2023rouquier,hanlon2024resolutions,ballard2024king}. 
We refer to \cite{hanlon2024resolutions,ballard2024king} for more extensive discussions of the definition of $\Theta$. However, we note that $-\Theta$ has image in the effective cone $\mathrm{Eff}(X)$ of $X$ which is the support of the secondary fan $\Sigma_{GKZ}(X)$ which encodes the toric birational geometry of $X$. 

In this setting, our \cref{prop:bijection}, which is the first sentence of \cref{thm:main}, follows from combining \cite{bondal2006derived} and \cite{achinger2015characterization} and is implicit in \cite{hanlon2024resolutions,ballard2024king}, where the stratification by toric hyperplanes is at times referred to as the Bondal stratification.\footnote{Strictly speaking, this claim is only true when $\Sigma$ contains at least one smooth cone so that $\varphi^{-1}(\ZZ^k) = \ZZ^n$ in \cref{thm:main} as the toric literature works with the Bondal stratification on $M_\RR/M$. In general, \cref{thm:main} instead concerns a stratification on a torus that admits a finite-sheeted cover by $M_\RR/M$.} 
However, \cref{thm:main} not only makes this correspondence explicit and purely combinatorial, but it also gives new insight on the combinatorial relationship between the Bondal stratification and the half-open zonotope $Z$ as in the following corollary.

\begin{maincor} \label{thm:maintoric} Suppose that $\tau$ is the minimal face of $\mathrm{Eff}(X)$ containing the image of $d \in -\Theta$. Then, 
\[ \dim \tau + \dim S_d  = k - | J_{S_d} | \]
where $S_d$ is the $\Phi$-stratum of the Bondal stratification corresponding to $-d$.
Moreover, $\tau$ is naturally identified with the effective cone of a $\dim(S_d)$ toric variety whose Bondal-Thomsen collection is in bijection with the subset of $\Theta_X$ corresponding to strata $S$ with $J_{S_d} \subseteq J_S$.  
\end{maincor} 

 We prove \cref{thm:maintoric} in \cref{sec:toric}, where we also make a few observations about additional consequences and related concepts. Note that both \cref{ex:hirz2} and \cref{ex:blhirz2} come from semiprojective toric varieties with simplicial fans and hence \cref{thm:maintoric} applies. 
 We leave the following question for future research.

 \begin{question}
    In order to resolve a toric subvariety in \cite{hanlon2024resolutions}, one restricts the Bondal stratification to a subtorus of $M_\RR/M$. 
    It would be interesting to extend \cref{thm:main} and \cref{thm:maintoric} to this setting.
    That is, what are the dimensions of the induced strata corresponding to Bondal-Thomsen elements that appear in the resolution of a given toric subvariety?
 \end{question}

\subsection{Acknowledgements}
We thank Jeff Hicks, Oleg Lazarev, and all the authors of \cite{ballard2024king} for useful conversations that shaped our view of the combinatorics of the Bondal-Thomsen collection.
Additionally, we thank the anonymous referee for comments that improved the paper.
This paper resulted from the SHUR program for undergraduate research at Dartmouth College in Summer 2024.
We are grateful to the other SHUR organizer (Ina Petkova) and participants (Ashton Lewis, Zachary Okajli, and Benjamin Shapiro)  for contributing to a stimulating summer research environment.

A.H., D.P., S.S., and B.S. were supported by National Science Foundation grant DMS-2404882.
The SHUR program was supported by Dartmouth College and by National Science Foundation CAREER grant DMS-2145090.

\section{Proof of \cref{thm:main}} \label{sec:proof}
In this section, we prove \cref{thm:main} in several steps. 
Before analyzing dimensions, we start with showing that there is a bijection between lattice points of the half-open zonotope and $\Phi$-strata.
\begin{prop} \label{prop:bijection}
The map $\pi \circ \Phi$ induces a bijection between $\Phi$-strata of $\RR^n/\varphi^{-1}(\ZZ^k)$ and lattice points of $Z$.
\end{prop}

\begin{proof}
Since $\img(\Phi) \subseteq \ZZ^k$, it is clear that $\pi \circ \Phi$ has image in $\pi(\ZZ^k)$. 
Also, if $u \in \RR^n$, then the $j$th coordinate of $\Phi(u)$ is $\lceil \langle u, v_j \rangle \rceil = \langle u, v_j \rangle + c_j$ where $c_j \in [0,1)$.
That is, $\Phi(u) = \varphi(u) + c$ where $c \in [0,1)^k$, and we see that the image of $\pi \circ \Phi$ also lies in $Z$.
Since $\Phi$ is constant on $\Phi$-strata, we have shown that $\pi \circ \Phi$ induces a map from $\Phi$-strata to $\pi(\ZZ^k) \cap Z$. 

We now show that the map is injective. Assume that for $u, v \in \RR^n$, we have $\pi(\Phi(u)) = \pi(\Phi(v))$. This implies that $\Phi(u) - \Phi(v) = \varphi(x)$ for some $x \in \RR^n$. However, since $\img(\Phi) \subseteq \ZZ^k$, we note that $\Phi(u)-\Phi(v) = \varphi(x) \in \ZZ^k$, i.e., $x \in \varphi^{-1}(\ZZ^k)$. Then, we note that
\[ \Phi(v+x) = \lceil\varphi(v+x)\rceil = \lceil\varphi(v)\rceil+\varphi(x) = \Phi(v)+\varphi(x)\]
which allows us to rewrite our earlier assumption as 
\[\Phi(u) = \Phi(v+x).\]
Thus, $u$ and $v+x$ are in the same pre-image of $\Phi$. Since $x \in \varphi^{-1}(\ZZ^k)$, $u$ and $v$ are in the same $\Phi$-stratum and thus the map is injective.

Next, we show surjectivity. Let $p \in Z \cap \pi(\ZZ^k)$.
Since $\pi$ is surjective and $Z = \pi([0,1)^k)$, there exists $x \in [0,1)^k$ such that $p = \pi(x)$. Thus, 
\[ \pi^{-1}(p) = \{ x + \varphi(y) \mid y \in \RR^n \}\]
Moreover, as $p \in \pi(\ZZ^k)$, there exists $m \in \ZZ^k$ such that $m \in \pi^{-1}(p)$. Thus, there exists a $y \in \RR^n$ such that
\[x + \varphi(y) = m.\]
However, as $x \in [0, 1)^k$, the only integral value $x+\varphi(y)$ can take is $\lceil\varphi(y)\rceil$. Thus, we have that
\[m = \lceil\varphi(y)\rceil = \Phi(y)\]
showing $\pi(\Phi(y)) = p$.
We have therefore deduced that the map is surjective. 
\end{proof}

It remains to verify \eqref{eq:dimension}, that the minimal face of $Z$ containing the image of a $\Phi$-stratum is determined by $J_S$, and that this relation is inclusion-reversing.
We will use a lemma from convex geometry.

\begin{lem}[{\cite[Theorem 6.6]{rockafellar1997convex}}] \label{lem:relint}
    Let $C \subseteq \RR^n$ be a convex set and let $f\colon \RR^n \to \RR^m$ be a linear map. Then,
    \[\relint(f(C)) = f(\relint(C))\]
    where $\relint$ denotes the relative interior. 
\end{lem}

Using \cref{lem:relint}, we will identify a face of the cube in $\RR^k$ that projects to the minimal face of the zonotope containing the image of a $\Phi$-stratum. 
Namely let $S$ be a $\Phi$-stratum and $p = \pi(\Phi(S)) \in Z$ the corresponding point under \cref{prop:bijection}. 
We set $F_p$ to be the minimal face of $Z$ containing $p$ and 
\begin{equation} \label{eq:fs} 
    F_S = \left\{ x \in [0,1)^k \mid x_j = 0 \text{ for all } j \in J_S \right\} 
\end{equation}
where $J_S$ is as in \eqref{eq:js}.
These two faces are related by the following lemma.

\begin{lem} \label{lem:face}
    $F_p = \pi(F_S)$. 
\end{lem}
\begin{proof}
    We will show that $\pi^{-1}(F_p) \cap [0,1)^k = F_S$ from which the statement follows as $[0,1)^k$ surjects onto $Z$ under $\pi$.  Since $\pi$ is a linear map, it follows directly from the definition of face of a convex set that $\pi^{-1}(F_p) \cap [0,1)^k$ is a face of $[0,1)^k$. Denote this face by $G_p$ and note that as a face of $[0,1)^k$, we have
    \[ G_p \coloneq \{ x \in [0, 1)^k : x_j = 0 \text{ for all } j \in I \}\]
    for some $I \subset \{1, \hdots, k\}$. 
    Our claim then reduces to showing that $I = J_S$. 

    First, we show $I \subseteq J_S$. 
    Let $y \in \wt{S}$ for some lift $\wt{S}$ of $S$ and $x = \Phi(y)-\varphi(y) \in [0,1)^k$.
    Since $\pi(x) = p$, we have that $x \in G_p$. 
    Then, we have
    \begin{align*}
        j \in I &\implies x_j = 0 \\
        &\iff \varphi(y)_j = \langle y, v_j \rangle  \in \ZZ \\
        &\iff h_j(y) = 0 \\
        &\iff j \in J_S.
    \end{align*}
    Therefore, $I \subseteq J_S$. 
    
    For the reverse inclusion, we note that $p \in \relint(F_p)$.
    By applying \cref{lem:relint} with $C = G_p$ and $f = \pi$, we find that there is a point $z \in \relint(G_p)$ such that $\pi(z) = p$.
    Then, if $y \in \wt{S}$ for some lift $S$, we see that
    \[ \Phi(y)-z \in \ker(\pi) = \img(\varphi), \]
    that is, there is a $y' \in \RR^n$ such that $\Phi(y) - z = \varphi(y')$. Next, note that as $z \in [0, 1)^k$ and $\Phi(y) \in \ZZ^k$, we have 
    \[ \Phi(y') = \lceil\varphi(y')\rceil = \lceil\Phi(y)-z\rceil = \lceil\Phi(y)\rceil = \Phi(y). \]
    Thus, $y' \in \wt{S}$ and $\Phi(y') - \varphi(y') = z \in G_p$. As a consequence, we have
    \begin{align*}
        j \in J_S &\implies \langle y', v_j \rangle = \varphi(y')_j \in \ZZ \\ 
        &\iff z_j = 0 \\
        &\iff j \in I
    \end{align*}
    where the last equivalence follows from $z \in \relint(G_p)$. 
    Therefore, we have also deduced that $J_S \subseteq I$.
\end{proof}

With \cref{lem:face}, we are ready to prove our main result. 

\begin{proof}[Proof of \cref{thm:main}] 
    In \cref{prop:bijection}, we showed that $\pi \circ \Phi$ induces a bijection between the $\Phi$-strata of $\RR^n/\varphi^{-1}(\ZZ^k)$ and $\pi(\ZZ^k) \cap Z$.
    Let $S$ be a $\Phi$-stratum and $p = \pi(\Phi(S))$ lie in a minimal face $F_p$ of $Z$. 
    By \cref{lem:face}, $F_p = \pi(F_S)$ where $F_S$ is as in \eqref{eq:fs}.
    It follows immediately that $F_p$ is entirely determined by $J_S$ and this assignment is inclusion-reversing.
    
    It remains to prove \eqref{eq:dimension}. From \cref{lem:face} and the rank-nullity theorem applied to $\pi$ restricted to the span of $F_S$, we have
    \[ \dim(F_S) = \dim(F_p) + \dim(\ker(\pi)\cap F_S) = \dim(F_p) + \dim(\img(\varphi)\cap F_S). \]
    From its definition, we know that $\dim(F_S) = k - |J_S|$, and it remains to analyze $\dim(\img(\varphi)\cap F_S)$. 
    We claim that $\img(\varphi) \cap \mathrm{span}(F_S) = \varphi(\mathrm{span}(\wt{S}-u))$ for any lift $\wt{S}$ and $u \in \wt{S}$.
    Assuming that, applying the rank-nullity theorem again to $\varphi$ restricted to $\mathrm{span}(\wt{S}-u)$, we obtain
    \[ \dim(S) =  \dim(\img(\varphi)\cap F_S) + \dim\left(\ker \varphi \cap \mathrm{span}(\wt{S}-u)\right) \]
    from which \eqref{eq:dimension} follows.
    
    To conclude, let us verify our claim that $\img(\varphi) \cap \mathrm{span}(F_S) = \varphi(\mathrm{span}(\wt{S}-u))$. 
    That is, we need to show that 
    \[ \mathrm{span}(\wt{S}-u) = \left\{ w \in \RR^n \mid \langle w, v_j \rangle = 0 \text{ for all } j \in J_S \right\}. \]
    Denote the set on the right-hand side by $W_S$. The fact that $\mathrm{span}(\wt{S}-u) \subseteq W_S$ is immediate from the definition of $J_S$.  
    On the other hand, $\wt{S} -u$ is the intersection of a finite number of closed and open half-spaces and $W_S$ for which none of the half-spaces contain $\wt{S} -u$ in their boundary.
    That is, the closure of $\wt{S}-u$ is a full-dimensional polytope in $V$, and thus its span coincides with $W_S$.
\end{proof}

\section{Faces from restriction} \label{sec:faces}

In this section, we will further analyze the faces of $Z$ and show that \cref{thm:main} is compatible with restricting to certain lower-dimensional toric hyperplane stratifications.

Given a $\Phi$-stratum $S$, we define
\[ W_S \coloneq \left\{ w \in \RR^n \mid \langle w, v_j \rangle = 0 \text{ for all } j \in J_S  \right\} \]
as in the proof of \cref{thm:main} and consider the commutative diagram 
\begin{equation} \label{eq:restrictdiagram} 
    \begin{tikzcd}
        \RR^n \arrow{r}{\varphi} & \RR^k  \arrow{r}{\pi} & \mathrm{coker}(\varphi)\\%
        W_S \arrow{r}{\tilde{\varphi}} \arrow[hookrightarrow]{u} & \RR^{k-|J_S|} \arrow{r}{\tilde{\pi}} \arrow{u}{\iota} & \mathrm{coker}(\tilde{\varphi}) \arrow{u}{\tilde{\iota}}
    \end{tikzcd}
\end{equation}
where $\iota$ is the inclusion of the subspace where all coordinates in $J_S$ are zero, $\tilde{\varphi}$ is given by pairing with the $v_j$ for $j \not \in J_S$, $\tilde{\pi}$ is the projection to the cokernel of $\tilde{\varphi}$, and $\tilde{\iota}$ is induced by $\iota$. 

From the bottom row of \eqref{eq:restrictdiagram}, we get another stratification on $W_S/\tilde{\varphi}^{-1}(\ZZ^{k-|J_S|})$ and a zonotope $\wt{Z}$ in $\mathrm{coker}(\tilde{\varphi})$. 
We will now show that $\wt{Z}$ is naturally identified with the minimal face of $Z$ containing the lattice point corresponding to $S$.

\begin{prop} \label{prop:restrictedface}
    Let $p = \pi(\Phi(S))$ be the lattice point in $Z$ corresponding to $S$ in \cref{prop:bijection}, and, as before, let $F_p$ be the minimal face of $Z$ containing $p$. The map $\tilde{\iota}$ is injective and $\tilde{\iota}(\wt{Z}) = F_p$. 
\end{prop}
\begin{proof} Suppose that $0 = \tilde{\iota}(\tilde{\pi}(x)) = \pi(\iota(x))$. 
Then, $\iota(x) = \varphi(u)$ for some $u \in \RR^n$. 
However, $\iota(x)_j = 0$ for $j \in J_S$ so $u \in W_S$, and we see that $\tilde{\pi}(x) = \tilde{\pi}(\tilde{\varphi}(u)) = 0$ showing that $\tilde{\iota}$ is injective.

Now, note that $\iota([0,1)^{k-|J_S|}) = F_S$ where $F_S$ is as in \eqref{eq:fs}. Applying \cref{lem:face} and the commutativity of \cref{eq:restrictdiagram}, we obtain that $\tilde{\iota}(\wt{Z}) = F_p$.
\end{proof}

\begin{rem} 
Note that $S$ becomes a codimension zero $\tilde{\Phi}$-stratum in $W_S$. 
Moreover, \cref{thm:main} has a more straightforward proof for codimension zero strata.
Namely, if $S$ has codimenzion zero, then there is a point $u \in S$ such that $\varphi(u)$ has no integral coordinates and hence $\Phi(u) - \varphi(u) \in (0,1)^k$. 
Applying the finite-dimensional case of the open mapping theorem, $\pi$ must map $(0,1)^k$ to the interior of $Z$.
Thus, an alternative path to \cref{thm:main} is to treat \cref{prop:restrictedface} as the main intermediary lemma. 
\end{rem}

\begin{rem} 
\cref{prop:restrictedface} and \eqref{eq:restrictdiagram} are compatible with the inclusion-reversing nature of the correspondence between $J_S$ and $F_p$. 
In particular, if $J_{S_1} \subseteq J_{S_2}$, then $W_{S_2} \subseteq W_{S_1}$, and there is a commutative diagram analogous to \eqref{eq:restrictdiagram} with a row for each of $S_1, S_2$. 
\end{rem}

\begin{example}
 In \cref{ex:hirz2}, there is a $\Phi$-stratum $S$ with $J_S = \{2,4\}$ so the orthogonal vectors are $\pm e_2$. $S$ is labeled in red in \Cref{fig:p1}. 
 Identifying $W_S = \{ u_2 = 0 \} \subset \RR^2$ with $\RR$ via the $u_1$ coordinate, the zonotope $\tilde{Z}$ corresponds to taking the set of vectors $\tilde{A} = \{ \pm 1 \}$.
 We can see $\tilde{Z}$ as a face of $Z$ in the rightmost image in \Cref{fig:p1}. 
\end{example}

\section{Boundary Bondal-Thomsen elements} \label{sec:toric}

In this section, we prove \cref{thm:maintoric} and make a few related observations.
We continue with the notation from \cref{subsec:toricintro}. 

\begin{proof}[Proof of \cref{thm:maintoric}]
Since the boundary of $Z_X$ lies on the boundary of the effective cone, $\tau$ coincides with the span of the minimal face $F$ of $Z_X$ containing the lattice point corresponding to $S_d$. Thus, the first part of the claim follows immediately from \cref{thm:main} after noting that $\ker \varphi = 0$ since we have assumed the $X$ is semiprojective and hence $\Sigma(1)$ contains a basis of $N_\RR$. 

For the second part, we consider \eqref{eq:restrictdiagram}.
The dual lattice to $W_{S_d} \subseteq M$ is 
\[ N_{S_d} \coloneq N/\mathrm{span}_{\ZZ}(u_\rho \mid \rho \in J_{S_d}) \]
where $u_\rho$ is the primitive generator of $\rho \in \Sigma(1)$. 
Then, we obtain a collection of rays in $N_{S_d}$ as the image of $\Sigma(1)\setminus \{\rho \mid \rho \in J_{S_d} \}$ which determines a secondary fan (or equivalently a toric GIT problem). 
Let $Y$ be any of the toric varieties corresponding to a chamber in the moving cone of that secondary fan.
Then, \cref{prop:restrictedface} identifies the half-open zonotope $Z_Y$ of $Y$ with $F$.
Therefore, the Bondal-Thomsen collection of $Y$, which is in correspondence with lattice points of $Z_Y$, is in bijection with the subset of $\Theta_X$ lying on $F$, which are precisely those coming from strata $S$ with $J_{S_d} \subseteq J_S$ by \cref{thm:main}. 
Now, since the span of $Z_Y$ is $\mathrm{Eff}(Y)$ and the span of $F$ is $\tau$, we have a natural identification of $\mathrm{Eff}(Y)$ and $\tau$. 
\end{proof} 

When Bondal-Thomsen elements lie on the boundary of the effective cone as in \cref{thm:maintoric}, there are homological consequences. 
For example, the following statement follows immediately from \cref{thm:maintoric} and \cite[Proposition 2.25]{ballard2024king}. 

\begin{prop} \label{prop:pushforward}
Suppose that $\tau$ is a face of $\mathrm{Eff}(X)$ and $Y$ is a toric variety with effective cone naturally identified with $\tau$ as in \cref{thm:maintoric}. If $X'$ corresponds to a chamber of $\Sigma_{GKZ}(X)$ adjacent to $\tau$, then
\[ (R\pi_*)\mathcal{O}_{X'}(-d) = 0 \]
for all $d \in -\Theta$ whose image does not lie in $\tau$ and where $\pi$ is the map induced by viewing the fan of $X'$ as a refinement of a generalized fan for $Y$. 
\end{prop}

Although we have chosen to state \cref{thm:maintoric} and frame our discussion thus far in terms of a chosen toric variety $X$, the careful reader will note that all our constructions depend only on $\Sigma(1)$ and thus are actually invariants of the Cox ring of $X$.
Categorically, that means the natural repository for any homological consequences is the Cox category $D_{Cox}(X)$ which was recently introduced in \cite{ballard2024king}.
For example, with a little more work, one can show that an analogue of \cref{prop:pushforward} holds for a natural functor from $D_{Cox}(X)$ to $D_{Cox}(Y)$ though we do not prove it here, as setting up the Cox categories and the functor would take us much further from the elementary combinatorics which are the focus of this paper.
Instead, we briefly continue our speculation and note that the analogue of \cref{prop:pushforward} on Cox categories would imply
\begin{equation} \label{eq:orthogonal}
    \mathrm{Hom}_{D_{Cox}(X)}(\mathcal{O}_{Cox}(-d), \mathcal{O}_{Cox}(-d')) = 0 
\end{equation}
when the dimension of the minimal face of $Z_X$ corresponding to $d \in -\Theta$ is smaller than that corresponding to $d' \in -\Theta$ and where we refer to \cite[Definition 4.1]{ballard2024king} for the definition of  $\mathcal{O}_{Cox}(-d)$.
For example, there are no morphisms from the Bondal-Thomsen generator corresponding to the red dot in \Cref{fig:p1} to any of the Bondal Thomsen generators corresponding to the blue dots in \Cref{fig:p1} in the Cox category of the Hirzeburch surface of type two.
It should also be possible to verify \eqref{eq:orthogonal} under the dimension condition above directly using \cite[Corollary 4.25]{ballard2024king}.  
To summarize, the combinatorics of \cref{thm:main} demonstrates that the geometry of $Z$ is related to semi-orthogonality in the Cox category.

We end by pointing out a homological problem on toric varieties that seems to require a more subtle combinatorial analysis.
Lower-dimensional $\Phi$-strata in the Bondal stratification are closely related to linear subvarieties in the sense of \cite[Definition 5.8]{hanlon2024resolutions} as both require a subset of $\Sigma(1)$ that spans a subspace of $N$. 
However, the presence of a linear subvariety depends also on the higher dimensional cones.
In particular, \cite[Conjecture 5.11]{hanlon2024resolutions} posits a relationship between linear subvarieties of $X$ and the line bundles whose Frobenius pushforwards generate the derived category.  
Although our results may provide some insight, this problem requires additional understanding of where $X$ sits in its secondary fan and cannot be solved entirely from the combinatorics of $Z$ or a construction that lifts to the Cox category.

\printbibliography

\Addresses

\end{document}

%% file: preambles/preamble.tex
\usepackage{amsmath,amsthm,amsfonts,amssymb,csquotes,epsfig, titlesec, epsfig, float, caption,wrapfig, mathrsfs}
\usepackage[hidelinks]{hyperref}
\usepackage{tikz-cd}
\usepackage{tikz}
\usepackage{tikz-3dplot}
\usepackage{xcolor}
\usepackage{verbatim}
\usepackage{accents}
\usepackage{marvosym}
\usepackage{multirow}
\usepackage{multicol}
\usepackage{listings}
\usepackage[citestyle=alphabetic,bibstyle=alphabetic,backend=bibtex, maxbibnames=10,minbibnames=5]{biblatex}
\usepackage[colorinlistoftodos]{todonotes}
\usepackage[]{enumerate}
\usepackage[american]{babel}

\usetikzlibrary{calc, decorations.pathreplacing,shapes.misc}
\usetikzlibrary{decorations.pathmorphing}
\usepackage[left=1in,top=1in,right=1in]{geometry}
\usepackage[capitalize]{cleveref}
\usepackage{subcaption}
\usetikzlibrary{shapes.geometric}

\newtheorem{thm}{Theorem}[section]

\newtheorem{lem}[thm]{Lemma}

\newtheorem{prop}[thm]{Proposition}

\theoremstyle{remark} 
\newtheorem{rem}[thm]{Remark}

 \crefname{rem}{Remark}{Remarks}
 \Crefname{rem}{Remark}{Remarks}

\newtheorem{example}[thm]{Example}

\newtheorem{question}[thm]{Question}

\theoremstyle{definition}

\titleformat*{\section}{\normalsize \bfseries \filcenter}
\titleformat*{\subsection}{\normalsize \bfseries }
\newtheorem{mainthm}{Theorem}
\Crefname{mainthm}{Theorem}{Theorems}
\newtheorem{maincor}[mainthm]{Corollary}
\Crefname{maincor}{Corollary}{Corollaries}

\Crefname{mainquestion}{Question}{Questions}

\Crefname{mainconj}{Conjecture}{Conjectures}

\crefformat{equation}{(#2#1#3)}

\makeatletter
\def\namedlabel#1#2{\begingroup
   \def\@currentlabel{#2}
   \label{#1}\endgroup
}
\makeatother

\DeclareFontFamily{U}{mathx}{}
\DeclareFontShape{U}{mathx}{m}{n}{<-> mathx10}{}
\DeclareSymbolFont{mathx}{U}{mathx}{m}{n}
\DeclareMathAccent{\widehat}{0}{mathx}{"70}
\DeclareMathAccent{\widecheck}{0}{mathx}{"71}

\renewbibmacro{in:}{}

%% file: preambles/mathpreamble.tex
\newcommand{\wt}{\widetilde}

\newcommand{\RR}{\mathbb R}
\newcommand{\ZZ}{\mathbb Z}

\DeclareMathOperator{\img}{im}
\DeclareMathOperator{\relint}{relint}